\documentclass{article}
\usepackage{lmodern}

\usepackage{amssymb,amsmath,amsfonts,amsthm,epsfig,pstricks,graphics,tikz}

\usepackage{hyperref}

\usepackage[width=0.85\textwidth ]{caption}

\newtheorem{theorem}{Theorem}[section]

\usepackage{hyperref}

\providecommand{\keywords}[1]
{
  \small	
  \textbf{\textit{Keywords---}} #1
}

\title{Gravitational billiards bouncing inside general domains - foci curves and confined domains}

\author{Daniel Jaud\\Orcid: 0000-0002-0163-7586\\
E-Mail: Daniel.Jaud.PhD@gmail.com\\
Gymnasium Holzkirchen}

\date{}
\begin{document}
\maketitle
\flushbottom


\begin{abstract}
\noindent A massive particle under the influence of a constant gravitational force that is bouncing inside an ideal reflecting mirror described by some function $f(x)$ is considered. For the associated flight trajectories we derive the parametric curves, named foci curves. All foci points of the parabolas for a given initial position and energy lie on these curves. From these foci curves the associated flight parabola envelopes are derived resulting, together with the mirror surface, in a confined domain for all possible particle trajectories in the non-periodic orbit case. The general results are briefly discussed and visualized for three concrete mirror surfaces.
\end{abstract}

\vspace*{0.4cm}


\section{Introduction}\label{sec:Introduction}
In mathematical billiards the trajectory of a point-like particle inside a domain with ideal reflecting boundary is studied. During the last years such billiard systems, for particles propagating along straight lines before hitting the boundary and being reflected according to the law of reflection, have become of special interest. Generalizations to classical billiard type systems under the influence of a constant gravitational force e.g.  have also been studied via an analytic \cite{Anderson,Costa2015}, a geometric \cite{Masalovich2014,Masalovich2020} or a numerical approach \cite{Lehihet1986} for different scenarios. Also simple quantum mechanical systems under the influence of a gravitational force have recently been studied regarding e.g. the particle's probability distribution in \cite{Jaud2022}.

In the present work we study the confined domains for a point-like massive particle within an ideal reflecting mirror, whose boundary is described by some function $f(x)$, under the influence of a homogeneous gravitational field. We will derive a parametric expression for the curve on which all consecutive flight parabola foci points, starting from some initial focus, will lie. In particular we proof that this so called foci curve is totally determined by the geometry of the mirror boundary in which the particle is subjected to bounce. As a further result, we derive the parametric expressions for the two occurring envelopes of the flight parabolas, sharing the same foci curve and total energy, and show that these envelopes together with the mirror boundary confine the possible particle trajectory in the non-periodic orbit case.

The structure of this paper is as follows. In Section \ref{sec:generalities} general facts about the motion of a particle in a homogeneous gravitational field are motivated. Section \ref{sec:foci_curves} deals with the derivation of the parametric curve on which consecutive flight parabola foci lie. The associated confined domains are in general derived in Section \ref{sec:envelopes}. Three specific examples will be discussed in Section \ref{sec:examples} in order to illustrate the general results derived before.

\section{Generalities for billiards in a gravitational field}
\label{sec:generalities}

In this section we briefly review some general results on the parabolic motion of a massive particle subjected to a homogeneous gravitational field.

A particle starting from the origin of the coordinate system with some velocity vector $\vec{v}=(v_x,v_y)$ describes in the presence of a constant gravitational force a parabolic arc. The particle's trajectory depending on the traveling time can be decomposed in a motion along the $x-$ and $y-$direction separately according to
\begin{align}
x(t)&=v_xt,\\
y(t)&=-\frac{1}{2}gt^2+v_yt.
\end{align}
Eliminating time this results in associated trajectory
\begin{equation}
y(x)=-\frac{g}{2v_x^2}\cdot x^2 +\frac{v_y}{v_x}\cdot x,
\end{equation}
which by completing the square can be brought  in the form 
\begin{equation}
y(x)=-\frac{g}{2v_x^2}\cdot \left(x-\frac{v_xv_y}{g}\right)^2+\frac{v_y^2}{2g}.
\end{equation}

Comparing this to a general parabola given in focal form 
\begin{equation}
y=-\frac{1}{4f}\cdot (x-x_F)^2+y_F+f,
\end{equation}
\begin{figure}[htb]
\centering
\includegraphics[scale=0.8]{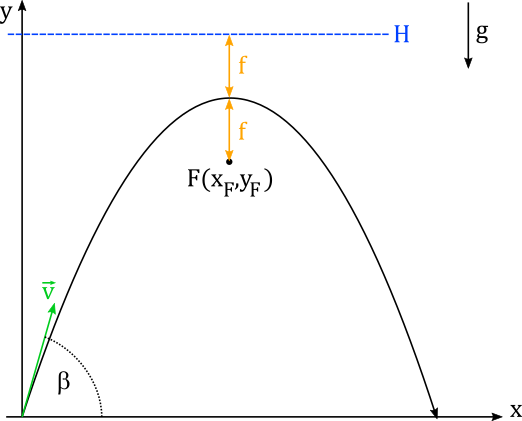}
\caption{Important quantities characterizing a possible flight parabola.}
\label{fig:parabola_generalities}
\end{figure}

see Fig. \ref{fig:parabola_generalities}, we conclude that the flight parabola (as we call it) has a focal length $f$ related to the velocity at the origin in the $x-$direction or equivalently described via the absolute value $v=|\vec{v}|$ and the angle $\beta$ made with the horizontal 
\begin{equation}
f=\frac{v_x^2}{2g}=\frac{v^2\cdot \cos^2(\beta)}{2g}.
\end{equation}
Similar the coordinates of the flight parabola focus $F$ are given by
\begin{equation}\label{eq:focus_parabola}
(x_F,y_F)=\left(\frac{v_xv_y}{g},\frac{v_y^2-v_x^2}{2g}\right)=\left(\frac{v^2\cdot \sin(2\beta)}{2g},-\frac{v^2\cdot \cos(2\beta)}{2g}\right).
\end{equation}

By conservation of energy it is easy to verify that all flight parabolas with the same total energy $E$ have a common directrix line $y=H=\frac{E}{mg}$, where $m$ is the mass of the particle (see Fig. \ref{fig:parabola_generalities}). In general the relation between the $y-$direction focus coordinate $y_F$, the directrix $H$ and the focal length $f$ reads 
\begin{equation}
H=y_F+2f.
\end{equation}
We will use the last equation in Section \ref{sec:envelopes} in order to derive the parametric curves describing the envelopes of the flight parabolas within a general reflecting mirror domain.

A further important result, e.g. derived in \cite{Masalovich2014,Masalovich2020}, states that the foci $F$ \& $F'$ of two flight parabolas connected to the same point of reflection $A$ (see Fig. \ref{fig:foci_circle}) lie on a common circle of radius $R_0=\frac{v^2}{2g}$. Here $v$ describes the absolute value of the velocity vectors of the two flight parabolas at the point of reflection $A$. Note that due to the law of reflection for the incident velocity $\vec{v}$ and excident velocity $\vec{v}'$ at $A$ holds $|\vec{v}|=|\vec{v}'|=v$.

\begin{figure}[htb]
\centering
\includegraphics[scale=1.17]{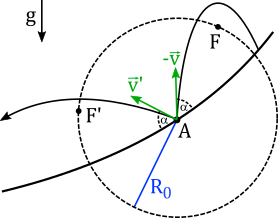}
\caption{Foci of consecutive flight parabolas lie on a common circle.}
\label{fig:foci_circle}
\end{figure}

We will make intensive use of this property in the upcoming section resulting in a formula for calculating flight parabola envelopes for general mirror boundaries.

As a last ingredient the law of reflection implies, for a particle hitting the boundary of a mirror with slope (at the point of reflection ) characterized by the angle $\alpha$ with respect to the horizontal, the following relation between the in- and outgoing directions of the flight parabolas (see Fig. \ref{fig:reflection})

\begin{equation}
\beta'=\pi+2\alpha -\beta.
\end{equation}

Plugging this relation into the expression for the focus point, see Eq. \eqref{eq:focus_parabola}, yields the associated linear mapping between $F$ and $F'$ according to

\begin{equation}
\begin{pmatrix}
x_{F'}\\
y_{F'}
\end{pmatrix}=\begin{pmatrix}
-\cos(4\alpha) & -\sin(4\alpha)\\
-\sin(4\alpha) & \cos(4\alpha)
\end{pmatrix}
\begin{pmatrix}
x_F\\
y_F
\end{pmatrix}.
\end{equation}
It is straight forward to verify that the linear map implies that two consecutive focus points for the flight parabolas related by the same point of reflection lie on a common straight line with slope
\begin{equation}
m=\frac{y_{F'}-y_F}{x_{F'}-x_F}=\tan(2\alpha).
\end{equation}

\begin{figure}[htb]
\centering
\includegraphics[scale=0.9]{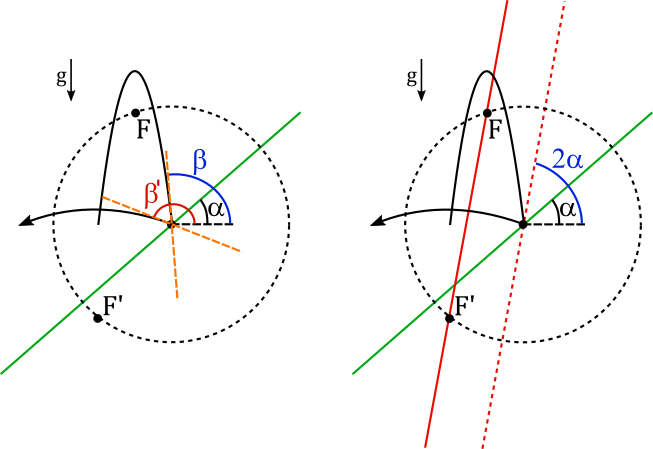}
\caption{Angles for law of reflection (\textit{left}). Flight parabola foci $F$ and $F'$ sharing the same point of reflection lie on a common line of angle $2\alpha$  if at the point of reflection the mirror tangent has an angle $\alpha$ with respect to the horizontal (\textit{right}).}
\label{fig:reflection}
\end{figure}

Concluding the consecutive focus point can geometrically be constructed by the intersection of a circle with radius $R_0=\frac{v^2}{2g}$, where $v$ is the absolute value of the velocity vectors at the point of reflection, and the straight line of slope $m=\tan(2\alpha)$ passing through $F$ (see Fig. \ref{fig:reflection}). Note that the slope of the intersection line does depend on the slope of the mirror at the point of reflection and thus varies for non-planar mirrors.

\section{The foci curves}\label{sec:foci_curves}

The goal of this section is to derive a parametric expression for the curve on which all foci of the flight parabola billiard system lie for a general mirror boundary described by some function $f(x)$. As reviewed in the last section two consecutive flight parabola foci lie along a common circle centered at the point of reflection. We are now considering the following setup: given an ideal reflecting mirror described by $f(x)$ and a point-like particle of mass $m$. If the particle has some total energy $E$ then at any given point of reflection along the mirror boundary $(k,f(k))$ the associated energy can be written as (note $v^2=\vec{v}^2$)
\begin{equation}
E=\frac{m}{2}v^2+mg\cdot f(k).
\end{equation}
Dividing by $mg$ we can associate the expression $E/mg$ with a maximal reachable height $L$ of the particle, i.e.
\begin{equation}\label{eq:L_heigt}
L=\frac{E}{mg}=\frac{v^2}{2g}+f(k).
\end{equation}
For the radius $R(k)$ of two consecutive foci points with common point of reflection $(k,f(k))$ thus holds

\begin{equation}
R(k)=\frac{v^2}{2g}=L-f(k).
\end{equation}
The set of all circles of radius $R(k)$ with center $(k,f(k))$ along the mirror boundary form a domain (see Fig. \ref{fig:general_construction}), whose points are given via the parametric expression
\begin{equation}\label{eq:coordinates_domain}
(x(k,\vartheta),y(k,\vartheta))=(k+[L-f(k)]\cdot \cos(\vartheta),f(k)+[L-f(k)]\cdot \sin(\vartheta)),
\end{equation}
where $\vartheta \in [0;2\pi]$. It is clear that all possible flight parabola foci have to lie within this domain. 

\begin{figure}[htb]
\centering
\includegraphics[scale=0.7]{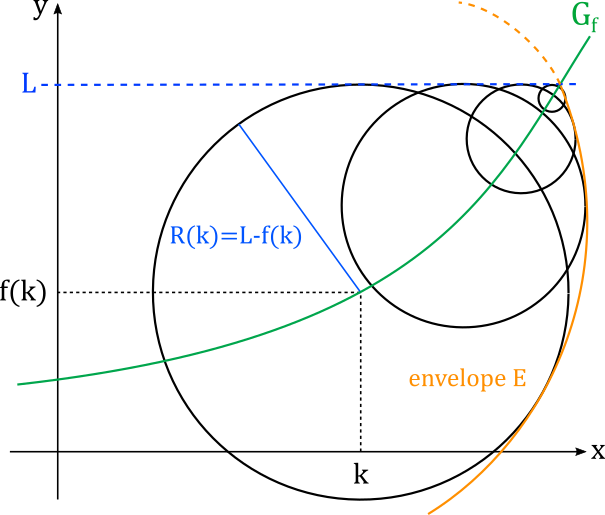}
\caption{Geometric construction of the maximal foci envelope $E$.}
\label{fig:general_construction}
\end{figure}

The domain (see Fig. \ref{fig:general_construction}) is confined by a non-trivial envelope curve $E$ which can be determined by solving the conditional equation \cite{Yates2012}
\begin{equation}{\label{eq:parametric_cond_1}}
\frac{\partial x}{\partial k}\cdot \frac{\partial y}{\partial \vartheta}-\frac{\partial y}{\partial k}\cdot \frac{\partial x}{\partial \vartheta}=0,
\end{equation}
for the variable $\vartheta$. Plugging Eq. \eqref{eq:coordinates_domain} in the conditional equation \eqref{eq:parametric_cond_1} this is equivalent to
\begin{equation}
[L-f(k)]\cdot [\cos(\vartheta)+f'(k)\cdot \sin(\vartheta)-f'(k)]=0.
\end{equation}
Solutions to the last equation are given by
\begin{equation}
\vartheta =\frac{\pi}{2}~~~~~\mbox{or}~~~~~\vartheta =\arctan\left(\frac{2f'(k)}{1+f'(k)^2},\frac{f'(k)^2-1}{f'(k)^2+1}\right),
\end{equation}
where $\vartheta=\frac{\pi}{2}$ simply corresponds to the cases where the foci lie on the straight line $f(k)=L$ which is clear from the construction but of no further relevance in our considerations. By definition $\arctan(a,b)$ gives the arc tangent of $\frac{b}{a}$, taking into account which quadrant the point $(a,b)$ is in \cite{Wolfram}.
Using the well-known identities
\begin{align}
\cos(\arctan(a,b))&=\frac{a}{\sqrt{a^2+b^2}},\\
\sin(\arctan(a,b))&=\frac{b}{\sqrt{a^2+b^2}},
\end{align}
the envelope curve thus can be written in parametric form as 
\begin{equation}{\label{eq:foci_envolute}}
(x(k),y(k))=\left(k+\frac{[L-f(k)]\cdot 2f'(k)}{1+f'(k)^2},f(k)+\frac{[L-f(k)]\cdot (f'(k)^2-1)}{1+f'(k)^2}\right).
\end{equation}
So far we have found a parametric expression for the, as we call it, envelope curve of the foci circles along the mirror boundary depending on an additional parameter $L$ associated with a maximal reachable height of the particle. Note, as indicated in Fig. \ref{fig:general_construction} by the dashed orange curve, the parametric expression for the envelope yields an analytic continuation for values where $f(k)>L$ which will be of physical relevance later on. This envelope curve given by Eq. \eqref{eq:foci_envolute} has now a practical implication stated in the following theorem.

\begin{theorem}{\label{thm:foci_curve}}
Given some ideal reflecting mirror boundary described by some function $f(x)$ and an initial particle position with flight parabola focal coordinates $(x_{F_0},y_{F_0})$ and associated focal length $f=\frac{v_x^2}{2g}$. Then holds: If the parameter $L$ of the envelope curve given by Eq. \eqref{eq:foci_envolute} is matched such that for a given $k$ holds $(x(k),y(k))=(x_{F_0},y_{F_0})$, then all consecutive flight parabola foci points lie on the same envelope curve with fixed $L$. In that case we call the envelope curve the foci curve denoted by $(x_F(k),y_F(k))$.
\end{theorem}

\begin{proof}
For the proof consider Fig. \ref{fig:proof_main}.
\begin{figure}[!htb]
\centering
\includegraphics[scale=0.9]{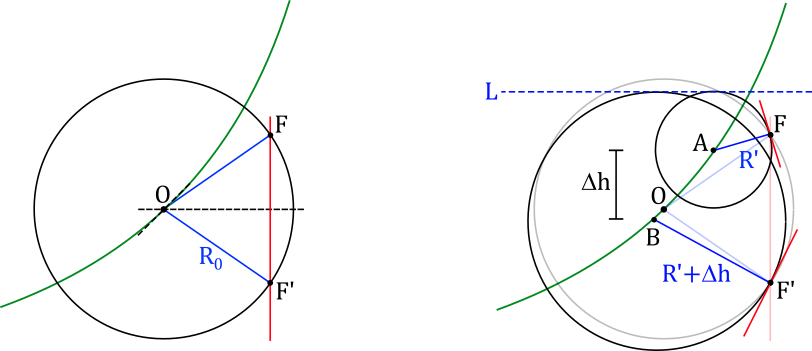}
\caption{Dual description of the possible flight parabola foci points for consecutive reflections.}
\label{fig:proof_main}
\end{figure}

Let $O$ be the point of reflection of two consecutive flight parabolas and $F$ be an initial focus point associated to one of the two parabolas. Then by the general construction procedure the consecutive focus point $F'$ is obtained by the intersection of the circle of radius $R_0$ centered at $O$ and the straight line passing through $F$ related to the slope of the mirror at $O$ (compare section \ref{sec:generalities}). 
In an equivalent description we find a value $L$ such that for a given point $A$ on the mirror the focus $F$ lies along the circle of radius $R'$ tangent to $y=L$ and the intersection line (red line see Fig. \ref{fig:proof_main} \textit{right}) also being tangent to this circle. In addition there exists another point $B$ with vertical distance $\Delta h$ with respect to $A$ with corresponding radius $R'+\Delta h$ and tangent intersection line (red) on which $F'$ must lie. This means that the circle around $B$ with radius $R'+\Delta h$ is also tangent to $y=L$ and in particular $F$ and $F'$ lie by construction on the foci curve described by $(x_F(k),y_F(k))$ for some fixed value of $L$. Since the point $O$ and initial focus point $F$ were chosen arbitrarily this construction must be true for every consecutive focus point thus proofing that all related flight parabola foci points lie along one common foci curve for some initially fixed value of $L$.
\end{proof}

The theorem yields a mighty result since the curve on which all foci points of the gravitational billiard system lie, is totally governed by the geometry of the ideal reflecting mirror. Thus by evaluating Eq. \eqref{eq:foci_envolute} and matching $L$ to the particle's initial (focus) condition yields the foci curve for all possible mirror reflections.

\section{Flight parabola confined domains}\label{sec:envelopes}

Having determined an expression for the foci curve of the flight parabolas in the last section we are now in the position to calculate the envelope curve restricting the confined domain in which the particle for a given initial focus position $(x_{F_0},y_{F_0})$ and energy $E$, related to a common directrix line $y=H$, will bounce. 

\begin{theorem}{\label{thm:envolute}}
Given a foci curve $(x_F(k),y_F(k))$ of the flight parabolas. The envelope curves $(x_E(k),y_E(k))$, restricting the allowed domain for all possible flight parabolas with common directrix line $H$, are given by
$$\left(x_F(k)+(H-y_F(k))\cdot J_\pm (k), \frac{H}{2}\cdot (1-J_\pm^2(k))+\frac{y_F(k)}{2}\cdot (1+J_\pm^2(k))\right),$$
where
$$J_\pm (k):= \frac{x_F'(k) \pm\sqrt{x_F'(k)^2+y_F'(k)^2}}{y_F'(k)}.$$
\end{theorem}

\newpage
\begin{proof}
The parametric curves on which the flight parabolas of focal length $f$ and focus $(x_F(k),y_F(k))$ lie, are given by
\begin{equation}
(x(k,t),y(k,t))=(x_F(k)+t,y_F(k)+f-\frac{t^2}{4f}),
\end{equation}
where the parameter $L$ appearing in $(x_F(k),y_F(k)$ (see Eq. \eqref{eq:foci_envolute}) has to be adjusted to the foci position of the initial flight parabola as stated in Theorem \ref{thm:foci_curve}. Energy conservation restricts all flight parabolas to have a common directrix line $H$ such that (compare section \ref{sec:generalities})
\begin{equation}
y_F(k)+2f=H~~~~\rightarrow ~~~~ f=\frac{H-y_F(k)}{2}.
\end{equation}
Inserting this equality in the above expression the curve of all possible flight parabolas as a function of $k$ and $t$ with common directrix $H$ is given by
\begin{equation}{\label{eq:envelope_parametric}}
(x(k,t),y(k,t))=\left(x_F(k)+t,\frac{H+y_F(k)}{2}-\frac{t^2}{2(H-y_F(k))}\right).
\end{equation}
The envelope curve is obtained by solving the conditional expression Eq. \eqref{eq:parametric_cond_1} in the variables $(k,t)$ again resulting in two solutions
\begin{equation}
t_\pm =(H-y_F(k))\cdot \left[\frac{x_F'(k)}{y_F'(k)}\pm \frac{\sqrt{x_F'(k)^2+y_F'(k)^2}}{y_F'(k)}\right]=:(H-y_F(k))\cdot J_\pm (k).
\end{equation}
Inserting the last result $t_\pm$ in Eq. \eqref{eq:envelope_parametric} yields the stated result for the envelope curves.
\end{proof}

It is interesting to note that for the flight parabolas in general, two distinct envelopes corresponding to $J_\pm(k)$ arise, which may intersect and such further restrict the allowed flight domain. In addition, it should be clear that the mirror described by $f(x)$ gives a natural boundary intersecting with the envelopes. For a graphical representation of the obtained results consider Fig. \ref{fig:envelopes}.

\begin{figure}[htb]
\centering
\includegraphics[scale=0.9]{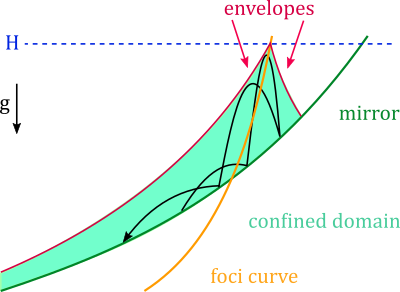}
\caption{Flight parabolas with common directrix line and same foci curve are restricted to a domain confined by two envelope curves and the mirror as a boundary.}
\label{fig:envelopes}
\end{figure}

\newpage
\section{Discussion and visualization of special cases}\label{sec:examples}

In this section, in order to illustrate our general results given by theorems \ref{thm:foci_curve} \& \ref{thm:envolute}, we discuss the corresponding foci curves and confined domains for three specific mirror boundaries. First a short review of the well-studied parabolic mirror, followed by the bouncing of a particle along a straight line tilted by some angle $\alpha$ and at last the case of a hyperbolic shaped mirror.

\subsection{Parabolic mirror revisited}
We start with a parabolic mirror of boundary given by $f(x)=\frac{1}{4f_m}x^2$. Here $f_m$ describes the focal length of the mirror. From Eq. \eqref{eq:foci_envolute} follows that the foci curve is given by
\begin{equation}
(x(k),y(k))=\left(k+\frac{[L-\frac{k^2}{4f_m}]\cdot \frac{k}{f_m}}{1+\frac{k^2}{4f_m^2}},\frac{k^2}{4f_m}+\frac{[L-\frac{k^2}{4f_m}]\cdot (\frac{k^2}{4f_m}-1)}{1+\frac{x^2}{4f_m^2}}\right).
\end{equation}

Defining
\begin{align}
\cos(\varphi)&=\frac{\frac{k}{f_m}}{1+\frac{k^2}{4f_m^2}},\\
\sin(\varphi)&=\frac{\frac{k^2}{4f_m^2}-1}{1+\frac{k^2}{4f_m^2}},
\end{align}
it turns out, that the foci curve can be written as
\begin{equation}
(x(\varphi),y(\varphi))=(R\cdot \cos(\varphi),R\cdot \sin(\varphi)+f_m),
\end{equation}
where the radius in this case is given by $R=f_m+L$. For a given initial focus position related to some fixed value for $L$ all consecutive flight parabola foci lie on a common circle of radius $R$ with center at $(0,f_m)$ as depicted e.g. in Fig \ref{fig:parabola_foci_envelope}. This analytic approach reflects the results obtained geometrically in \cite{Masalovich2020}.
\begin{figure}[!htb]
\centering
\includegraphics[scale=0.5]{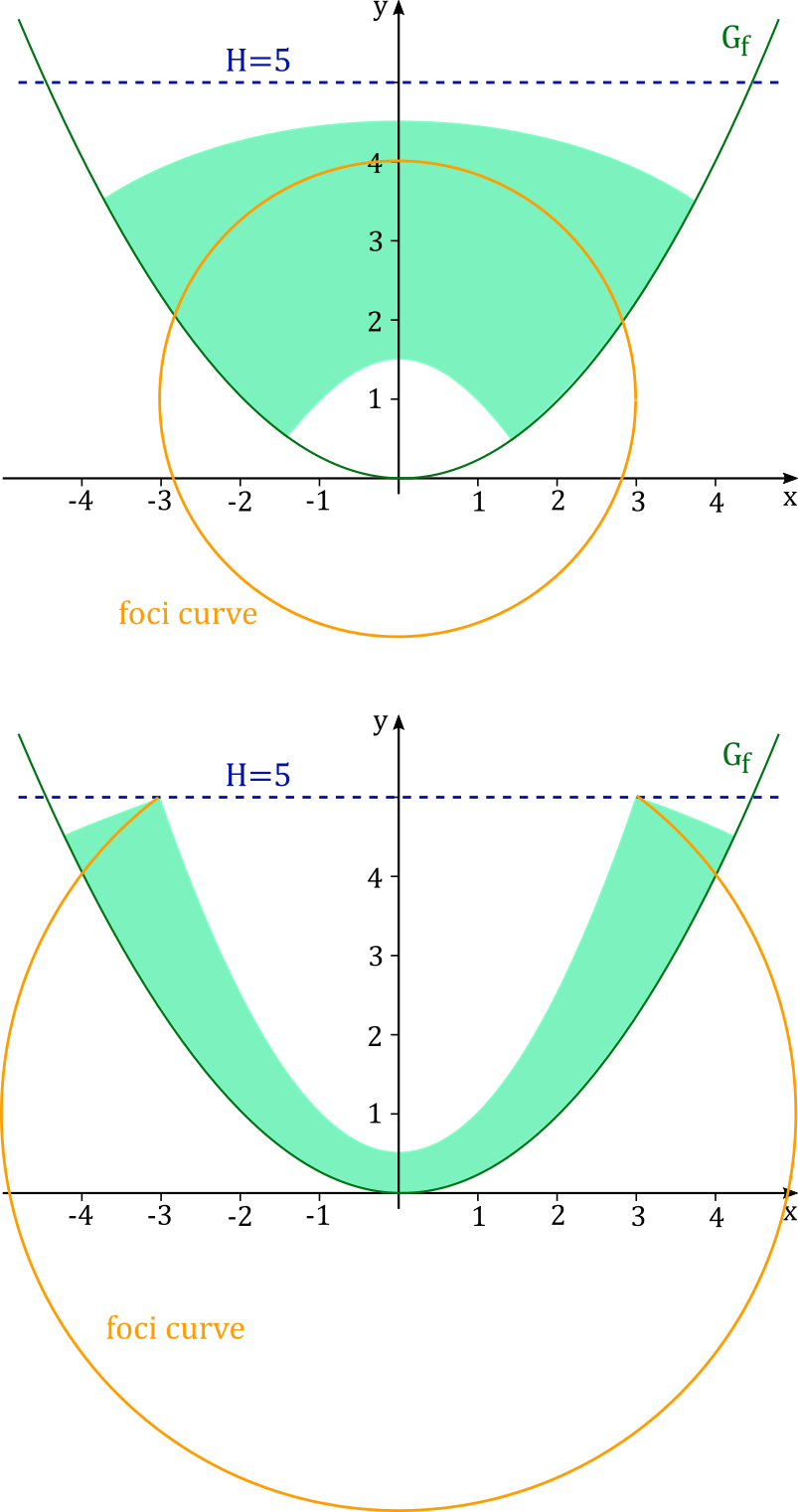}
\caption{Allowed flight parabola domains (green) obtained by the parabolic envelopes in cases $R=3$ (\textit{top}) and $R=5$ (\textit{bottom}) both for $H=5$ and parabolic mirror described by $f(x)=\frac{1}{4}x^2$ of focal length $f_m=1$.}
\label{fig:parabola_foci_envelope}
\end{figure}
A straight forward calculation (see Thm. \ref{thm:envolute}) including some redefinition of variables leads to the two envelope curves, confining the allowed particles flight domain, given by
\begin{equation}
(x_{E\pm}(k),y_{E\pm}(k))=\left(x,-\frac{x^2}{2(H-f_m\pm R)}+\frac{H+f_m\pm R}{2}\right),
\end{equation}
with $f(x)\leq H$ and consequently for the foci curve radius $0\leq R\leq H+f_m$. If the flight parabola performs a non-periodic orbit within the parabolic mirror, the trajectories are dense and swept out the entire domain between the mirror surfaces and the two envelopes corresponding to parabolas themselves (see Fig. \ref{fig:parabola_foci_envelope}).

Note that for $R\rightarrow 0$ the trajectory domain shrinks to a single parabola curve corresponding to a periodic orbit of the flight parabola.

\newpage
\subsection{Bouncing along a straight line}
As a second example we consider the case of the mirror being a straight line \newline ${f(x)=\tan(\alpha)\cdot x}$ of inclination angle $\alpha$ with respect to the horizontal (see Fig. \ref{fig:envelope_curve_line}).
\begin{figure}[htb] 
\centering
\includegraphics[scale=0.7]{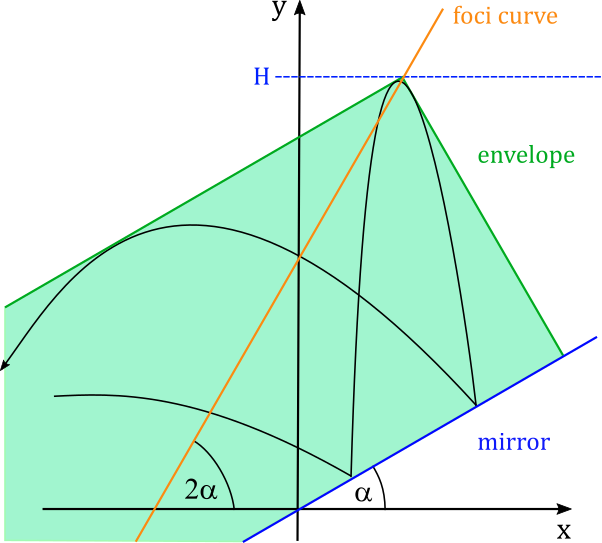}
\caption{Particle bouncing along a straight line mirror including foci curve as well as envelope curves resulting in a confined flight domain.}
\label{fig:envelope_curve_line}
\end{figure}

A direct calculation yields for the foci curve (compare Eq. \eqref{eq:foci_envolute})
\begin{equation}
(x_F(k),y_F(k))=\left(\frac{k\cdot (1-\tan^2(\alpha))+2L\cdot \tan(\alpha)}{1+\tan^2(\alpha)},\frac{2k\cdot \tan(\alpha)+L\cdot (\tan^2(\alpha)-1)}{1+\tan^2(\alpha)}\right).
\end{equation}
Via a suitable substitution of variables this can be rewritten in the form
\begin{equation}
(x_F,y_F)=(x,x\cdot \tan(2\alpha)-L\cdot \sec(2\alpha)) 
\end{equation}
corresponding to a straight line with slope $m=\tan(2\alpha)$ and $y-$intercept ${t=-L\cdot \sec(2\alpha)}$. For an initial flight parabola focus coordinate $(x_{F_0},y_{F_0})$, the parameter $L$ can be matched resulting in the final form
\begin{equation}
(x_F,y_F)=(x,x\cdot \tan(2\alpha)-x_{F_0}\cdot \tan(2\alpha)+y_{F_0}),
\end{equation}
for $\alpha \neq \frac{\pi}{2}$ and $x=x_{F_0}$ for $\alpha=\frac{\pi}{2}$.

%

Applying Thm. \ref{thm:envolute} the two corresponding flight parabola envelopes in terms of the initial flight parabola focus coordinate $(x_{F_0},y_{F_0})$ and directrix line $y=H$ read
\begin{equation}
(x_E,y_E)=(x,\pm (x-x_{F_0})\cdot \tan^{\pm 1}(\alpha)+\frac{y_{F_0}}{2}\cdot [1-\tan^{\pm 2}(\alpha)]+\frac{H}{2}\cdot [1+\tan^{\pm 2}(\alpha)]).
\end{equation}
Obviously, the envelopes correspond to two perpendicular straight lines of slope \newline $m_+=\tan(\alpha)$ and $m_-=-\cot(\alpha)=-\frac{1}{m_+}$  (see Fig. \ref{fig:envelope_curve_line}) intersecting along the directrix line with $y=H$. Note that the envelope line with $m_-$ represents the maximal allowed boundary of the flight parabola which, in general, must not be saturated by this setup.

The results obtained are in agreement with other approaches e.g. \cite{Anderson,Korsch1999}, where the envelope is calculated directly using the billiard map for the position and velocity of the flight parabolas in a wedge system or e.g. in \cite{Masalovich2020} for the special case of $\alpha=45^\circ$.

\newpage
\subsection{Bouncing inside a hyperbolic mirror}

As a last example, we briefly discuss the case of a hyperbolic mirror described by the function $f(x)=\sqrt{1+x^2}$. Obviously, we expect the confined flight parabola domain to be a kind of superposition between the two limiting cases where $f(x)\approx |x|$ for the large $x$ regime simply corresponding to a wedge system e.g. studied in \cite{Anderson} and $f(x)\approx 1+\frac{1}{2}x^2$ in the small $x$ limit resembling a parabolic mirror.

Direct calculation yields for the foci curve the parametric expression
\begin{equation}
(x_F(k),y_F(k))=\left(\frac{k\cdot (2L\sqrt{1+k^2}-1)}{1+2k^2},\frac{2\sqrt{1+k^2}^3-L}{1+2k^2}\right).
\end{equation}
In the large $k$ limit the curve approaches the two vertical asymptotes given by \newline $x_\pm =\pm L$. 

\begin{figure}[htb]
\centering
\includegraphics[scale=0.6]{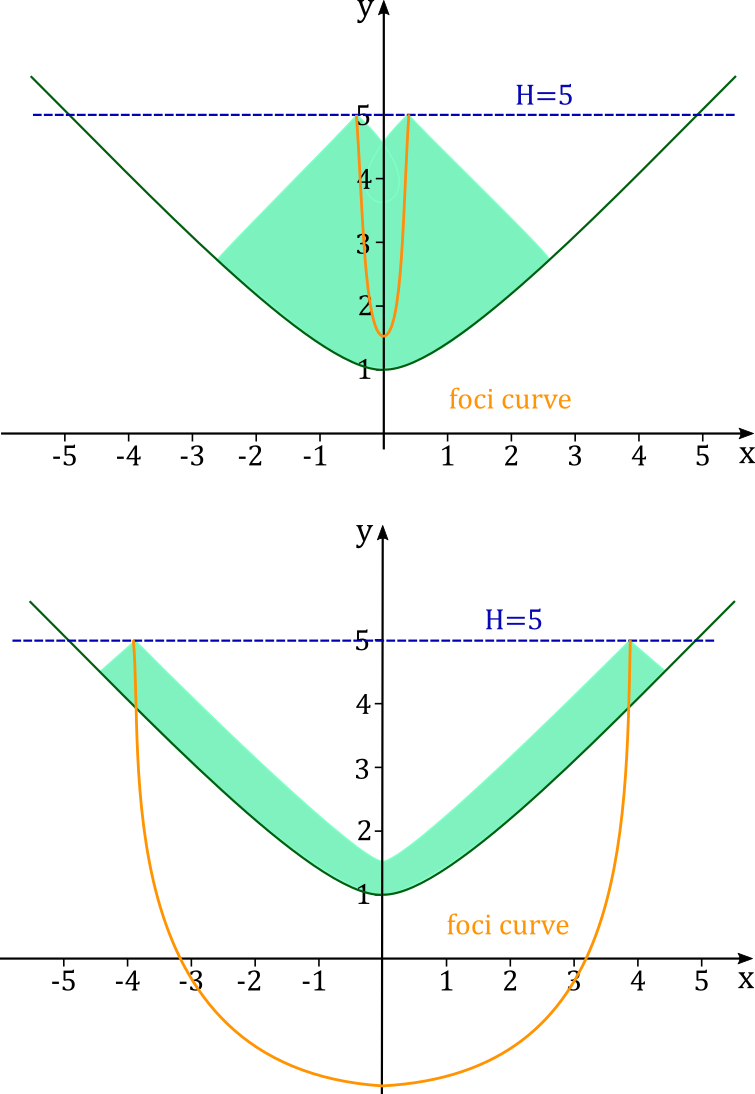}
\caption{Allowed flight parabola domains (green) obtained by the  envelopes in cases $L=0,5$ (\textit{top}) and $L=4$ (\textit{bottom}) both for $H=5$ and hyperbolic mirror described by $f(x)=\sqrt{1+x^2}$ .}
\label{fig:hyperbola_foci_envelope}
\end{figure}

Similar to the foci curve the associated envelope curves can be calculated according to Thm. \ref{thm:envolute}, where in this setup holds
\begin{equation}
J_\pm (k)=\frac{\sqrt{k^2+1}\cdot \left(2k^2-1\pm \sqrt{\frac{\left(2k^2+1\right)^2\left(4L\left(L+\sqrt{k^2+1}\left(2k^2-1\right)\right)+4k^6-3k^2+1\right)}{k^2+1}}\right)+2L}{2k\left(2L\sqrt{k^2+1}+2k^4+k^2-1\right)}.
\end{equation}

Again, considering the cases where the flight parabolas are dense within the given boundary, Fig \ref{fig:hyperbola_foci_envelope} shows (green domains) the swept out flight domains for the common directrix value $H=5$ but different initial starting positions associated with the parameter $L$.
As expected, the system behaves similar to the wedge case ($f(x)\approx |x|$) but with distinct non-accessible regions corresponding to the continuous deformation at the bottom of the mirror boundary.

\section{Conclusion and Outlook}\label{sec:conclusion}

We provided a novel approach to calculate the envelopes and associated confined domains of a bouncing particle under the influence of gravity being reflected along a 2D mirror described by some function $f$. Thereby, with our general results stated in theorems \ref{thm:foci_curve} \& \ref{thm:envolute}, we can reproduce and generalize the results for the confined domains already obtained in \cite{Anderson,Masalovich2014,Masalovich2020}. 

For future works e.g. it would be interesting to determine the mirror function $f(x)$ such that the foci curve represents an ellipse and the associated flight parabola vertex curve a circle.  This would implement a dual example to the parabolic mirror where the foci curve, as discussed, resembles a circle and the corresponding flight parabola vertex curve is given by an ellipse. Also, the precise mirror geometry corresponding to a parabolic foci curve would be of interest. From the results presented in this work it is clear that the associated mirror in the last case should be asymptotic to a straight line of slope 1.

Since all foci points have to lie along the foci curve despite a periodic or non periodic orbit it might be of interest the associated foci mapping $F_n\rightarrow F_{n+1}$ for consecutive flight parabolas. In the periodic case these should form a finite set. A study of their geometric behaviour and stability properties would be of interest.

Generalizations to non-constant gravitational forces, such as the gravitational field generated by a spherical mass distribution, could also be considered.

%
%



\vspace*{1cm}


\begin{thebibliography}{9}

\bibitem{Anderson}
K.D. Anderson, \textit{Dynamics of a Rotated Orthogonal Gravitational Wedge Billiard}, {\url{
https://doi.org/10.48550/arXiv.2206.04997}}



\bibitem{Costa2015}
D.R. da Costa, C.P. Dettmann \& E.D. Leonel, \textit{Circular, elliptic and oval billiards in a gravitational field}, Communications in Nonlinear Science and Numerical Simulation, Volume 22, Issues 1-3, May 2015, Pages 731-746, {\url{https://doi.org/10.1016/j.cnsns.2014.08.030}}

\bibitem{Jaud2022}
D. Jaud, \textit{Classical and quantum billiards inside the square with gravitational field}, {\url{
https://doi.org/10.13140/RG.2.2.26293.96487
}}

\bibitem{Korsch1999}
H.J. Korsch \& H.J. Jodl, \textit{Gravitational Billiards: The Wedge}, In: Chaos. Springer, Berlin, Heidelberg (1999), {\url{https://doi.org/10.1007/978-3-662-03866-6_4}}

\bibitem{Lehihet1986}
H.E. Lehihet \& B.N. Miller, \textit{Numerical study of a billiard in a gravitational field}, Physica D21 (1986) 93-104.

\bibitem{Masalovich2014}
S. Masalovich, \textit{A remarkable focusing property of a parabolic mirror for neutrons in the gravitational field: Geometric proof},
Nuclear Instruments and Methods in Physics Research Section A: Accelerators, Spectrometers, Detectors and Associated Equipment,
Volume 763,
2014,
Pages 517-520,
ISSN 0168-9002, \newline
{\url{https://doi.org/10.1016/j.nima.2014.07.004}}


\bibitem{Masalovich2020}
S. Masalovich, \textit{Billiards in a gravitational field: A particle bouncing on a parabolic and right angle mirror}, arXiv:2007.04730v2, \url{
https://doi.org/10.48550/arXiv.2007.04730}


\bibitem{Wolfram}
WolframAlpha {\url{https://reference.wolfram.com/language/ref/ArcTan.html}}


\bibitem{Yates2012}
R.C. Yates, \textit{Handbook on Curves and Their Properties}, Literary Licensing, LCC (2012)
\end{thebibliography}
\end{document}